\documentclass[11pt,reqno,a4paper]{amsart}

\begin{document}

\newcommand{\C}{\mathbb C}
\newcommand{\N}{\mathbb N}
\newcommand{\R}{\mathbb R}

\renewcommand{\Im}{{\mathrm{Im}}}
\renewcommand{\Re}{{\mathrm{Re}}}

\def\qqs{\ \forall \ }
\def \sprd#1,#2,#3{{\left( #1,#2\right)}_{#3}}
\newcommand{\lmd}{{\mathcal L}}

\newtheorem{theo}{Theorem}[section]
\newtheorem{prop}[theo]{Proposition}
\newtheorem{lem}[theo]{Lemma}
\newtheorem{cor}[theo]{Corollary}
\theoremstyle{definition}
\newtheorem{defi}[theo]{Definition}
\theoremstyle{remark}
\newtheorem{rem}[theo]{Remark}

\renewcommand{\cdots}{\dots}

\renewcommand{\theenumi}{\roman{enumi}}
\renewcommand{\labelenumi}{\theenumi)}

\title{Energy flow above the threshold of tunnel effect}

\author{F. Ali Mehmeti}

\address{%
Univ Lille Nord de France, F-59000 Lille, France \newline
\indent
 UVHC, LAMAV, FR CNRS 2956, F-59313 Valenciennes, France}

\email{felix.ali-mehmeti@univ-valenciennes.fr}

\author{R. Haller-Dintelmann}

\address{%
TU Darmstadt \\
Fachbereich Mathematik \\
Schlo{\ss}gartenstra{\ss}e 7\\
64289 Darmstadt\\
Germany}

\email{haller@mathematik.tu-darmstadt.de}

\author{V. R\'egnier}

\address{%
Univ Lille Nord de France, F-59000 Lille, France \newline
\indent
 UVHC, LAMAV, FR CNRS 2956, F-59313 Valenciennes, France}

\email{Virginie.Regnier@univ-valenciennes.fr}

\begin{abstract}
We consider the Klein-Gordon equation on two half-axes connected at
their origins. We add a potential that is constant but different on
each branch. In a previous paper, we studied the $L^{\infty}$-time
decay via H\"ormander's version of the stationary phase method. Here
we apply these results to show that for initial conditions in an
energy band above the threshold of the tunnel effect a fixed portion
of the energy propagates between group lines. Further we consider
the situation that the potential difference tends to infinity while
the energy band of the initial condition is shifted upwards such
that the particle stays above the threshold of the tunnel effect. We
show that the total transmitted energy as well as the portion
between the group lines tend to zero like $a_2^{-1/2}$ in the branch
with the higher potential $a_2$ as $a_2$ tends to infinity. At the
same time the cone formed by the group lines inclines to the
$t$-axis while its aperture tends to zero.
\end{abstract}

\subjclass[2000]{Primary 34B45; Secondary 47A70, 35B40\rm}

\keywords{Networks, Klein-Gordon equation, stationary phase method,
$L^{\infty}$-time decay, energy flow\rm}

\thanks{Parts of this work were done, while the second author visited
the University of Valenciennes. He wishes to express his gratitude to
F. Ali Mehmeti and the LAMAV for their hospitality}

\maketitle

%
%
%
%
\section{Introduction}
%
%
%
%
\noindent In this paper we study the energy flow of waves in
two coupled one-dimensional semi-infinite media having different
dispersion properties.   Results in experimental physics \cite{Nim,
hai.nim}, theoretical physics \cite{D-L} and functional analysis
\cite{Alreg3,yas} describe phenomena created in this situation by
the dynamics of the tunnel effect: the delayed reflection and
advanced transmission near nodes issuing two branches. Our purpose
is to describe the influence of the height of a potential step on
the energy flow of wave packets above the threshold of tunnel
effect.

We consider the following setting: let  $N_1, N_2$ be disjoint
copies of $(0,+ \infty)$. Consider numbers  $a_k$ satisfying   $0
\leq a_1 \leq a_2 < + \infty$. Find a vector $(u_1, u_2)$ of
functions $u_k: [0, +\infty) \times \overline{N_k} \rightarrow \C$
satisfying the Klein-Gordon equations
\[ [\partial_t^2 - \partial_x^2   + a_k ] u_k(t,x) = 0 , \ k = 1,2,
\]
on $N_1,N_2$ coupled at zero by usual Kirchhoff conditions and
complemented with initial conditions for the functions $u_k$ and
their derivatives.

Reformulating this as an abstract Cauchy problem, one is confronted
with the self-adjoint operator $A = (- \partial^2_x + a_1,
-\partial^2_x + a_2)$ in $L^2(N_1) \times L^2(N_2)$, with a domain
that incorporates the Kirchhoff transmission conditions at zero. For
an exact definition of $A$, we refer to Section~\ref{sec:sol:form}.
The problem described above can be reformulated as
\begin{equation} \label{abstract wave equation}
  \begin{array}{rcl}
    \ddot u(t) + A u(t) &=& 0, \\
    u(t) &\in& D(A),
  \end{array}
\end{equation}
for all $t > 0$ together with initial conditions. It is well known that the
following expression is invariant with respect to time for solutions
of \eqref{abstract wave equation}:
\begin{equation}\label{e-total energy}
    E(u(t,\cdot))= \frac{1}{2} \Bigl(\| \dot u(t,\cdot) \|_H^2
    + (Au,u)_H \Bigr).
\end{equation}
In Section~\ref{sec:sol:form} we recall the solution formula that
was proved in \cite{ahr.arxiv2} by an expansion in generalized
eigenfunctions in the more general setting of a star shaped network
with semi-infinite branches.

In Section~\ref{decay} we recall our result on $L^{\infty}$-time
decay proved in \cite{L-infinity-proc-IWOTA}. There we obtained the
exact $L^{\infty}$-time decay rate $c\cdot t^{-1/2}$ in the group
velocity cones together with an expression for the coefficient $c$
for initial conditions in a compact energy band in $(a_2,\infty)$.
In the present work we refine an estimation from below in this
context.

In Section~\ref{acc.decay} we use the preceding results to estimate
the $L^2$-norm of the outgoing solution on $N_2$, which is a part of
the total energy given in \eqref{e-total energy}. We suppose that
the initial condition belongs to an energy band above the threshold
of the tunnel effect, so that the solution propagates in the branch
with the higher potential $N_2$. We consider the $L^2$-norm of the
solution at time $t$ both on the whole branch $N_2$ and inside the
cone delimited by the group lines corresponding to the bounds of the
energy band. It turns out that the first norm has an upper
asymptotic bound and and the second one an upper and lower
asymptotic bound which behave as $a_2^{-1/2}$, if the height of the
potential step $a_2$ tends to infinity. This implies that the ratio
of the energy on the whole branch and the energy between the group
lines is time asymptotically confined in a finite interval above $1$
which is independent and $a_2$.

These results might be interpreted in terms of quantum mechanics as
follows: a relativistic, massive particle without spin in a
one-dimensional world is submitted to a potential step at the
origin. It is supposed to have enough kinetic energy to overcome the
step with a fixed remaining energy. Classically, the particle should
leave the potential step with a velocity which is independent of the
height of the step. Our results show that in the quantum mechanical
model, when the height of the potential step tends to infinity, the
velocity of the outgoing component of the particle tends to zero
while the particle is more and more localized. At the same time the
total outgoing energy tends to zero while its ratio to the energy
inside the cone delimited by the group lines remains time
asymptotically in a constant finite interval above one. Similar
estimates should be possible for the other parts of the total energy
\eqref{e-total energy}.

The last observation might lead to the idea of viewing the total
energy as being subdivided in a part inside the group line cones and
a part outside these cones. The ratio of these two energies is quite
independent of the experimental configuration, but depends mainly on
the chosen energy band.

The results of this paper are related to results in experimental
physics, theoretical physics and functional analysis (spectral
theory, asymptotic estimates, analysis on networks, cf.~%
\cite{L-infinity-proc-IWOTA} and the references cited there). For
example in \cite{Alreg3} we obtained information on the splitting of
the energy flow near zero. In \cite{fam2} a (not optimal) estimation
for the $L^{\infty}$-time decay rate has been obtained but without
any information on the localization of the energy. In \cite{vbl.evl}
the relation of the eigenvalues of the Laplacian in an
$L^{\infty}$-setting on infinite, locally finite networks to the
adjacency operator of the network is studied. In \cite{kostr}, the
authors consider general networks with semi-infinite ends. They give
a construction to compute some generalized eigenfunctions but no
attempt is made to construct explicit inversion formulas.

%
%
%
\section{A solution formula} \label{sec:sol:form}
%
%
%
%
\noindent The aim of this section is to recall the tools we used in
\cite{ahr.arxiv2} as well as the solution formula of the same paper
for a special initial condition and to adapt this formula for the
use of the stationary phase method
in the next section.
\begin{defi}[Functional analytic framework]  \label{def.A} ~
\begin{enumerate}
\item  Let $N_1,N_2$ be
identified with $(0, +\infty)$. Put $N := \overline{N_1} \times
      \overline{N_2}$, identifying the endpoints $0$.
\item Two transmission conditions are introduced:
      \begin{align*}
        \strut \text{($T_0$): } & (u_1, u_2) \in C(\overline{N_1}) \times
               C(\overline{N_2}) \text{ satisfies } u_1(0) = u_2(0).
        \intertext{This condition in particular implies that $(u_1, u_2)$ may
            be viewed as a well-defined function on $N$.}
        \text{($T_1$): } & (u_1, u_2) \in C^1(\overline{N_1}) \times
               C^1(\overline{N_2}) \text{ satisfies } \sum_{k=1}^2
               \partial_x u_k(0^+) = 0.
      \end{align*}
\item Define the Hilbert space $H = L^2(N_1) \times L^2(N_2)$ with scalar
      product
      \[ (u,v)_H = \sum_{k=1}^2 (u_k,v_k)_{L^2(N_k)}
      \]
      and the operator $A: D(A) \longrightarrow H$ by
      \[ \begin{aligned}
            D(A) &= \Bigl\{ (u_1, u_2) \in H^2(N_1) \times H^2(N_2) :
        (u_1,u_2) \text{ satisfies } (T_0) \text{ and }
        (T_1) \Bigr\}, \\
            A(u_1,u_2) &= (A_1u_1,A_2 u_2) = \Bigl(- \partial^2_x
            u_k + a_k u_k
                  \Bigr)_{k=1,2}.
         \end{aligned}
      \]
\end{enumerate}
\end{defi}

Note that, if $a_1 = a_2 =0$, then $A$ is the Laplacian in the sense of the
existing literature, cf.\@ \cite{vbl.evl, kostr}.
%
%
%
\noindent
\begin{defi}[Fourier-type transform $V$]  \label{def.V} ~
\begin{enumerate}
\item For $k = 1,2 $ and $\lambda \in \C$ let
\[ \xi_k(\lambda) := \sqrt{\lambda - a_k} \quad \text{as well as} \quad
    s_1 (\lambda):= - \frac{ \xi_2(\lambda)}{\xi_1(\lambda)} \quad \text{and} \quad
    s_2 (\lambda):= - \frac{ \xi_1(\lambda)}{\xi_2(\lambda)}.
\]
Here, and in all what follows, the complex square root is chosen in such a way
that $\sqrt{r \cdot e^{i \phi}} = \sqrt{r} e^{i \phi/2}$ with $r>0$ and $\phi
\in [-\pi,\pi)$.
\item For $\lambda \in \C$ and $j,k \in \{ 1, 2\}$, we
define generalized eigenfunctions $F_{\lambda}^{\pm,j}: N
\rightarrow \C$ of $A$ by $F_{\lambda}^{\pm,j}(x) :=
F_{\lambda,k}^{\pm,j}(x)$ with
\[ \left\{ \begin{aligned}
       F_{\lambda,k}^{\pm,j}(x) &= \cos(\xi_j( \lambda) x ) \pm i
       s_j( \lambda)   \sin(\xi_j( \lambda) x ), &  \text{for } k = j,\\
      F_{\lambda,k}^{\pm,j}(x) &= \exp(\pm i \xi_k(\lambda)x), &
          \text{for } k \neq j.
   \end{aligned} \right.
\]
for $x \in \overline{N_k}.$
\item For $l = 1, 2$ let
\[ q_l(\lambda) := \begin{cases}
            0, & \text{if } \lambda < a_l, \\
                    \frac{\xi_l(\lambda)}{|\xi_1(\lambda) + \xi_2(\lambda)|^2},
            & \text{if } a_l < \lambda.
                   \end{cases}
\]
\item Considering $q_1$ and $q_2$ as weights for our $L^2$-spaces, we set
$L^2_q := L^2((a_1, + \infty), q_1) \times L^2((a_2, + \infty), q_2)$.
The corresponding scalar product is
\[ (F, G)_q := \sum_{k=1}^2 \int_{(a_k, +\infty)} q_k(\lambda) F_k(\lambda)
    \overline{G_k(\lambda)} \; d \lambda
\]
and its associated norm $|F|_q := (F, F)_q^{1/2}$.
\item For all $f \in L^1(N, \C)$ we define $Vf: [a_1, + \infty) \times
      [a_2 , + \infty) \to \C$ by
\[ (Vf)_k(\lambda):= \int_N f(x) \overline{(F_{\lambda}^{-,k})}(x) \; dx, \ k = 1, 2.
\]
\end{enumerate}
\end{defi}
%
In \cite{ahr.arxiv2}, we show that $V$ diagonalizes $A$ and we
determine a metric setting in which it is an isometry. Let us recall
these useful properties of $V$ as well as the fact that the
property $u \in D(A^j)$ can be characterized in terms of the decay
rate of the components of $Vu$.
\begin{theo} \label{V.iso}
Endow $C_c^\infty(N_1) \times C_c^\infty(N_2)$ with the norm of $H =
L^2(N_1) \times L^2(N_2)$. Then
\begin{enumerate}
\item \label{V.iso:i} $V : C_c^\infty(N_1) \times C_c^\infty(N_2) \to L^2_q$ is
isometric and can be extended to an isometry $\tilde{V} : H \to L^2_q$, which
we shall again denote by $V$ in the following.
\item $V : H \to L^2_q$ is a spectral representation of $H$ with respect to $A$.
In particular, $V$ is surjective.
\item The spectrum of the operator $A$ is $\sigma(A) = [a_1, + \infty)$.
\item For $l \in \N$ the following statements are equivalent:
    \begin{enumerate}
    \item $u \in D(A^l)$,
    \item $\lambda \mapsto \lambda^{l} (Vu)(\lambda) \in L_q^2$,
    \item $\lambda \mapsto \lambda^{l} (Vu)_k(\lambda) \in
    L^2((a_k, + \infty),q_k), \ k=1,2$.
    \end{enumerate}
\end{enumerate}
\end{theo}
%
We are now interested in the Abstract Cauchy Problem
\[ \hbox{(ACP)}: u_{tt}(t) + Au(t) = 0, \ t > 0, \textrm{ with } u(0)=u_0,\ u_t(0)=0.
\]
Here, the zero initial condition for the velocity is just chosen for
simplicity, as we will not deal with the general case in this contribution.

By the surjectivity of $V$ (cf.\@ Theorem~\ref{V.iso}~(ii))  there
exists an initial condition $u_0 \in H$ satisfying

\medskip

\noindent {\bf Condition} ${\mathbf(A)}$:
  $(V u_0)_2 \equiv 0$ and $(V u_0)_1 \in C^2_c ((a_2, \infty))$.

\medskip

\begin{rem} \label{D A infty}
\begin{enumerate}
  \item For $u_0$ satisfying $(A)$ there exist
  $a_2 < \lambda_{\min} < \lambda_{\max}< \infty$
  such that
  \[\hbox{ supp} (V u_0)_1 \subset [\lambda_{\min} , \lambda_{\max}].\]
  \item
  If $u_0 \in H$ satisfies $(A)$, then
  $u_0 \in D(A^{\infty})= \displaystyle \bigcap_{l \geq 1} D(A^l)$,
  due to Theorem~\ref{V.iso}~(iv),
  since
  $\lambda \mapsto \lambda^{l} (Vu)_m(\lambda) \in
    L^2((a_m, + \infty),q_m), \ m=1,2$
  for all $l \in \N$ by the compactness of
  $\hbox{ supp} (V u_0)_m$.
\end{enumerate}
\end{rem}
%
%
\begin{theo}[Solution formula of (ACP) in a special case]
\label{solform.lambda}
Suppose that $u_0$ satisfies Condition $(A)$. Then there exists a
unique solution $u$ of $(ACP)$ with
$u \in C^{l}([0, + \infty), D(A^{m/2}))$
for all $l,m \in \N$. For $x \in N_2$ we have the representation
\begin{equation*}
u_2(t,x) = \frac{1}{2} \bigl( u_+(t,x) + u_-(t,x) \bigr)
\end{equation*}
with
\begin{equation} \label{representation}
u_\pm(t,x):=\displaystyle\int_{\lambda_{\min}}^{\lambda_{\max}}
e^{\pm i \sqrt{\lambda}t} q_1(\lambda) e^{-i \xi_2(\lambda)x}
(Vu_0)_1(\lambda) d \lambda.
\end{equation}
\end{theo}
%
\begin{proof}
Since $v_0 = u_t(0) = 0$, we have for the solution of (ACP) the
representation
\begin{equation*} 
u(t)=V^{-1} \cos(\sqrt{\lambda}t)Vu_0.
\end{equation*}
(cf.~for example  \cite[Theorem 5.1]{fam2}). The expression for
$V^{-1}$ given in \cite{ahr.arxiv2} yields the formula for $u_\pm$.
\end{proof}
%
%
\begin{rem}
A solution formula for arbitrary initial conditions which is valid
on all branches is available in \cite{ahr.arxiv2}. This general
expression is not needed in the following.
\end{rem}
%
%
%
%
%
\section{$L^{\infty}$-time decay}
\label{decay}

Next, we quote the result on $L^\infty$-time decay of solutions to the
problem (ACP) from \cite{L-infinity-proc-IWOTA}. Here, we only consider
special initial conditions that are localized in energy in a compact
interval contained in $(a_2, a_2 + 1)$. For these it is possible to
give very explicit estimates for all the constants that appear in a
asymptotic expansion of the solution to the order $t^{1/2}$.

\begin{theo} \label{sol.acc.decay}
 Let $0 < \alpha < \beta < 1$ and $\psi \in C^2_c((\alpha, \beta))$ with $\|
 \psi \|_\infty = 1$ be given. Setting $\tilde \psi (\lambda) := \psi(\lambda -
 a_2)$, we choose the initial condition $u_0 \in H$ satisfying $(V u_0)_2 \equiv
 0$ and $(V u_0)_1 = \tilde \psi$. Furthermore, let $u_+$ be defined as in
 Theorem~\ref{solform.lambda}.

 Then there is a constant $C(\psi, \alpha, \beta)$ independent of $a_1$ and
 $a_2$, such that for all $t \in \R^+$ and all $x \in N_2$ with
 \begin{equation} \label{cone-t-x}
  \sqrt{\frac{a_2 + \beta}{\beta}} \le \frac tx \le
    \sqrt{\frac{a_2 + \alpha}{\alpha}}
 \end{equation}
we have
 \[ \bigl| u_+(t,x) - H(t,x,u_0) \cdot t^{-1/2} \bigr| \le C(\psi, \alpha,
    \beta) \cdot t^{-1},
 \]
where
\begin{equation*} 
 H (t,x,u_0):=
e^{-i  \varphi (p_0,t,x)}
(2i\pi)^{1/2}a_2^{3/4}  \  h_1(t,x) \ h_2(t,x)\ (Vu_0)_1(a_2+ p_0^2)
\end{equation*}
with
\begin{align*}
  \varphi (p, t, x):= \sqrt{a_2 + p^2} t - p x, & \qquad
    p_0 := \sqrt{\frac{a_2 x^2}{ (  t^2 - x^2)}}, \\
  h_1(t,x):= \left( \frac{\left( t/x \right)^2}{  \left( t/x
\right)^2 -1 } \right)^{3/4}, & \qquad
h_2(t,x):= \frac{\sqrt{ (a_2 - a_1) \left(  (t/x)^2 - 1 \right) +
a_1}} {\Bigl(  \sqrt{ (a_2 - a_1) \left(   (t/x)^2 - 1 \right) +
a_2} + \sqrt{a_2}\Bigr)^2}.
\end{align*}
It holds further
 \[ \bigl| H(t,x,u_0) \bigr| \le
g(a_1,a_2,\beta) :=
 \sqrt{2\pi}
    \frac{\sqrt{\beta} (a_2 + \beta)^{3/4}}
    {\sqrt{a_2} \sqrt{a_2 - a_1 + \beta}}
\sim \sqrt{2 \pi \beta} \ a_2^{-1/4} \quad \text{as } a_2 \to
+\infty.
 \]
\end{theo}
\begin{proof}
This is contained in Equation (8) from the proof of Theorem~3.2 and in
Theorem~4.1 in \cite{L-infinity-proc-IWOTA}.
\end{proof}

\begin{rem}
\begin{enumerate}
\item
Note that \eqref{cone-t-x} is equivalent to
\begin{equation*} 
v_{\min} \leq v(t,x):= (t/x)^2 -1 \leq v_{\max} \ .
\end{equation*}
where $v_{\min}:= \dfrac{a_2}{\lambda_{\max} - a_2} =
\dfrac{a_2}{\beta}$ and $v_{\max}:= \dfrac{a_2}{\lambda_{\min} -
a_2} = \dfrac{a_2}{\alpha}$.
\item
For later use we define $p_{\min}:= \xi_2(\lambda_{\min})$ and
$p_{\max}:= \xi_2(\lambda_{\max})$.
\end{enumerate}
\end{rem}

%
\section{Energy flow}
\label{acc.decay}
\noindent

In this section we use the asymptotic expansion from section
\ref{decay} to estimate the outgoing solution from above and below
in the cones given by the group velocities corresponding to the
bounds of the energy band of the initial condition. This leads to
time independent asymptotic upper  (Theorem \ref{energy in cone
upper}) and lower (Theorem \ref{H lower estimate}) estimates of the
$L^2$-norm of the solution on the space interval inside the cones.

Further an upper estimate of the $L^2$-norm of the outgoing solution
on the whole branch $N_2$ is obtained using  Plancherel's theorem
(Theorem \ref{energy on branch}).

These informations lead to our main result (Theorem \ref{energy in
cone}), where we give an upper bound for the ratio of the $L^2$-norm
on the whole branch $N_2$ and the $L^2$-norm inside the cone, which
is asymptotically independent of the height of the potential step.

\begin{theo} \label{H lower estimate}
In the setting of Theorem \ref {sol.acc.decay} suppose that
$\psi(\mu) \ge m >0$ for $\mu \in [\alpha',\beta']$ with $\alpha <
\alpha' < \beta' < \beta$. Then we have
\begin{enumerate}
  \item the lower estimate for the coefficient of $t^{-1/2}$:

\begin{align*}
\bigl| H(t,x,u_0) \bigr| &\ge  f(a_2,\alpha,\beta)\cdot m \\
&: = \sqrt{2 \pi } \ a_2^{3/4}
\Bigl(\frac{\beta}{a_2}+1\Bigr)^{3/4} \frac{1}{\sqrt{a_2}}
\sqrt{\frac{a_2 - a_1}{\alpha} + 1}
\Biggl( \sqrt{\frac{a_2 - a_1}{\alpha} + 1} + 1 \Biggr)^{-2}
m.\\
&\sim \sqrt{2\pi\alpha} \ a_2^{-1/4} m ,
 \hbox{ as } a_2 \rightarrow \infty
\end{align*}
for all $(t,x)$ satisfying
 \[ \sqrt{\frac{a_2 + \beta'}{\beta'}} \le \frac tx \le
    \sqrt{\frac{a_2 + \alpha'}{\alpha'}} \ .
\]
  \item and the lower estimate for the solution
 \[
\forall \varepsilon > 0 \ \exists t_0 > 0 \ \forall t > t_0 \quad
\bigl| u_+(t,x) \bigr| \ge \bigl(f(a_2,\alpha,\beta\bigr) m -
\varepsilon) \ t^{-1/2}
 \]
for all $(t,x)$ satisfying
 \[ \sqrt{\frac{a_2 + \beta'}{\beta'}} \le \frac tx \le
    \sqrt{\frac{a_2 + \alpha'}{\alpha'}} \ .
 \]

\end{enumerate}

\end{theo}
\begin{proof}
 Note that it is always possible to choose the initial condition in the
 indicated way, thanks to the surjectivity of $V$,
 cf.\@ Theorem~\ref{V.iso}~ii).

\noindent \underline{(i)}:

\noindent
 Theorem \ref{sol.acc.decay} implies
 \[ \bigl| H(t,x,u_0) \bigr| = \sqrt{2\pi} a_2^{3/4} h_1(t,x) \bigl| h_2(t,x)
    \bigr| \cdot \| (Vu_0)_1 \|_\infty.
 \]
We estimate
 \begin{align*}
  \bigl| h_2(t,x) \bigr| &= \biggl|
    \frac{\sqrt{(a_2 - a_1) ((t/x)^2 - 1) + a_2}}%
    {\bigl( \sqrt{(a_2 - a_1) ((t/x)^2 - 1) + a_2} + \sqrt{a_2} \bigr)^2}
    \biggr| \\
    &\ge
    \frac{\sqrt{(a_2 - a_1) v_{\max} + a_2}}%
    {\bigl( \sqrt{(a_2 - a_1) v_{\max} + a_2} + \sqrt{a_2} \bigr)^2}
    \\
    &\ge
    \frac{\sqrt{(a_2 - a_1) \frac{a_2}{\alpha} + a_2}}%
    {\bigl( \sqrt{(a_2 - a_1) \frac{a_2}{\alpha} + a_2} + \sqrt{a_2} \bigr)^2}
    \\
    &= \frac{1}{\sqrt{a_2}}
    \frac{\sqrt{  \frac{a_2 - a_1}{\alpha} + 1}}%
    {\bigl( \sqrt{  \frac{a_2 - a_1}{\alpha} + 1} + 1 \bigr)^2}
    \\
&\sim
\frac{\sqrt{\alpha}}{a_2}  \hbox{ as } a_2 \rightarrow \infty
 \end{align*}
Here we used, that $b\mapsto \frac{b}{(b+c)^2}$ is decreasing for
$b > c \geq 0$.
Further
\begin{align*}
     \bigl| h_1(t,x) \bigr| &\geq\biggl(
\frac{
\frac{\lambda_{\max}}{\lambda_{\max}-a_2}
 }{
\frac{\lambda_{\max}}{\lambda_{\max}-a_2} -1
}
     \biggr)^{3/4}
\ = \ (\frac{\beta}{a_2}+1)^{3/4} \rightarrow 1 \hbox{ as } a_2
\rightarrow \infty
\end{align*}
This implies (i).

\noindent \underline{(ii)}:

\noindent

Using the lower triangular inequality we find that $\forall
\varepsilon > 0 \ \exists t_0 > 0 \ \forall t > t_0  :$
\begin{align*}
\bigl| u_+(t,x) \bigr| &\ge
\bigl| u_+(t,x) - H(t,x,u_0) \cdot t^{-1/2} + H(t,x,u_0) \cdot
t^{-1/2}\bigr| \\
&\geq
\Bigl|\bigl|H(t,x,u_0)\bigr| \cdot t^{-1/2}
-\bigl|u_+(t,x) - H(t,x,u_0) \cdot t^{-1/2}\bigr| \Bigr|
\\
&\geq
\bigl(f(a_2,\alpha,\beta\bigr) m - \varepsilon) \ t^{-1/2}
\end{align*}

for all $(t,x)$ satisfying
 \[ \sqrt{\frac{a_2 + \beta'}{\beta'}} \le \frac tx \le
    \sqrt{\frac{a_2 + \alpha'}{\alpha'}} \ .
 \]
This shows (ii)
\end{proof}

%
%
\begin{theo} \label{energy on branch}
Suppose the setting of Theorem \ref {sol.acc.decay}.  Then we have
\begin{equation*}
\| u_+(t,\cdot) \|_{L^2(N_2)}
\leq
\frac{\sqrt{\beta}}{\sqrt{a_2-a_1+\alpha} \sqrt[4]{\alpha}}
\|\psi\|_{L^2((\alpha, \beta))}
\sim
\frac{\sqrt{\beta}}{\sqrt[4]{\alpha}}
\|\psi\|_{L^2((\alpha, \beta))} a_2^{-1/2}
\hbox{ as } a_2 \rightarrow \infty
\end{equation*}

\end{theo}

\begin{proof}
 In the expression~\eqref{representation} for $u_+$ we substitute
 $p := \xi_2(\lambda)$. This yields
 \[ u_+(t,x) = 2 \int_{p_{\min}}^{p_{\max}} e^{i\sqrt{s_2 + p^2} t}
    q_1(a_2 + p^2) e^{-i p x} (V u_0)_1(a_2 + p^2) p \; d p.
 \]
 Interpreting this integral as a Fourier transform we find by the
 Plancherel theorem
 \[ \| u_+(t, \cdot) \|_{L^2(N_2)} \le \| u_+(t, \cdot) \|_{L^2(\R)} =
    \bigl\| p \mapsto p q_1(a_2 + p^2) (Vu_0)_1(a_2 + p^2)
    \chi_{[p_{\min}, p_{\max}]}(p) \bigl\|_{L^2(\R)}
 \]
 Now, we use that $\mathrm{supp}((Vu_0)_1)$ is contained in the interval
 $[a_2 + \alpha, a_2 + \beta]$, so only the range $\sqrt{\alpha} \le p
 \le \sqrt{\beta}$ is relevant. For these values of $p$ we find
 \[ p q_1(a_2 + p^2) = p
    \frac{\sqrt{a_2 - a_1 + p^2}}{(\sqrt{a_2 - a_1 + p^2} + p)^2} \le
    \frac{p}{\sqrt{a_2 - a_1 + p^2}} \le
    \frac{\sqrt{\beta}}{\sqrt{a_2 - a_1 + \alpha}}
 \]
 So, we have
 \[ \| u_+(t, \cdot) \|_{L^2(N_2)} \le
    \frac{\sqrt{\beta}}{\sqrt{a_2 - a_1 + \alpha}} \bigl\|
    (Vu_0)_1(a_2 + p^2) \bigr\|_{L^2((p_{\min}, p_{\max}))}
 \]
 and substituting back $\lambda = a_2 + p^2$ we obtain
 \[ \bigl\| (Vu_0)_1(a_2 + p^2) \bigr\|_{L^2((p_{\min}, p_{\max}))}^2 =
    \int_{\lambda_{\min}}^{\lambda_{\max}} \bigl| (V u_0)_1 (\lambda)
    \bigr|^2 \frac{d\lambda}{\sqrt{\lambda - a_2}} =
    \int_{a_2 + \alpha}^{a_2 + \beta} |\tilde \psi(\lambda)|^2
    \frac{d\lambda}{\sqrt{\lambda - a_2}}.
 \]
 Setting, finally, $\mu = \lambda - a_2$ we end up with
 \[ \| u_+(t, \cdot) \|_{L^2(N_2)} \le
    \frac{\sqrt{\beta}}{\sqrt{a_2 - a_1 + \alpha}} \Bigl( \int_\alpha^\beta
    |\psi(\mu)|^2 \frac{d \mu}{\sqrt{\mu}} \Bigr)^{1/2} \le
    \frac{\sqrt{\beta}}{\sqrt{a_2 - a_1 + \alpha} \sqrt[4]{\alpha}}
    \| \psi \|_{L^2((\alpha, \beta))}
   \qedhere
 \]
\end{proof}

\begin{theo} \label{energy in cone}
In the setting of Theorem \ref {sol.acc.decay} suppose that
$\psi(\mu) \ge m >0$ for $\mu \in [\alpha',\beta']$ with $\alpha <
\alpha' < \beta' < \beta$. Then we have
\begin{enumerate}
  \item the estimate from below: $\forall \varepsilon > 0 \ \exists t_0 > 0 \ \forall t > t_0
  \quad$
  \begin{equation*}
\bigl|\bigl|u_+(t,\cdot)\bigl|\bigl|_{L^2(I'_t)}^2
\geq
\Bigl( f(a_2,\alpha,\beta)m - \varepsilon \Bigr)^2
\Bigl(
\sqrt{\frac{\beta' }{a_2 + \beta'}}
-\sqrt{\frac{\alpha' }{a_2 + \alpha'}}
\Bigr)
\end{equation*}
where
\begin{equation*}
I'_t :=
    \Biggl[ \
t\sqrt{\frac{\alpha' }{a_2 + \alpha'}}
    \ , \
 t\sqrt{\frac{\beta' }{a_2 + \beta'}}
    \
    \Biggr]
\end{equation*}
and $f$ is defined in Theorem \ref {sol.acc.decay} (i).
%
%
\item the estimate for the ratio of the global and local energy:\\
$\forall \varepsilon > 0 \ \exists t_0 > 0 \ \forall t > t_0$
\begin{align*}
    \frac{
\bigl|\bigl|u_+(t,\cdot)\bigl|\bigl|_{L^2(N_2)}
    }{
\bigl|\bigl|u_+(t,\cdot)\bigl|\bigl|_{L^2(I'_t)}
    }
& \leq
\frac{\sqrt{\beta}\  \| \psi \|_{L^2((\alpha, \beta))}
}{\sqrt{a_1-a_2+\alpha}\sqrt[4]{\alpha}}
\Bigl( f(a_2,\alpha,\beta)m - \varepsilon\Bigr)^{-1}
\Bigl(
\sqrt{\frac{\beta' }{a_2 + \beta'}}
-
\sqrt{\frac{\alpha' }{a_2 + \alpha'}}
\Bigr)^{-1/2}
\end{align*}
  \item and
\begin{align*}
  &
\limsup_{t\rightarrow \infty}
   \frac{
\bigl|\bigl|u_+(t,\cdot)\bigl|\bigl|_{L^2(N_2)}
    }{
\bigl|\bigl|u_+(t,\cdot)\bigl|\bigl|_{L^2(I'_t)}
    }\\
& \leq
\frac{\sqrt{\beta}\  \| \psi \|_{L^2((\alpha, \beta))}
}{\sqrt{a_1-a_2+\alpha}\sqrt[4]{\alpha}}
\Bigl( f(a_2,\alpha,\beta)m \Bigr)^{-1}
\Bigl(
\sqrt{\frac{\beta' }{a_2 + \beta'}}
-
\sqrt{\frac{\alpha' }{a_2 + \alpha'}}
\Bigr)^{-1/2}
\\
&\sim
(2\pi)^{-1/2}m^{-1}\beta^{1/2}\alpha^{-3/4}
(\sqrt{\beta'}-\sqrt{\alpha'})^{-1/2}
\| \psi \|_{L^2((\alpha, \beta))}
\hbox{ as } a_2 \rightarrow \infty
\end{align*}

\end{enumerate}
\end{theo}
\begin{proof}

\noindent \underline{(i)}:

\noindent Follows from
\begin{equation*}
    \bigl|\bigl|u_+(t,\cdot)\bigl|\bigl|_{L^2(I'_t)}^2
\geq
\Bigl(\inf_{x \in I'_t}\bigl|u_+(t,\cdot)\bigl|\Bigr)^2
\cdot\bigl|I'_t\bigl|
\end{equation*}
and Theorem \ref{H lower estimate} (ii).

\noindent \underline{(ii)}:

\noindent Follows from (i) and Theorem \ref{energy on branch}.

\noindent \underline{(iii)}:

\noindent Direct consequence of (ii) and Theorem \ref{H lower
estimate} (i).
\end{proof}
\begin{theo} \label{energy in cone upper}
Consider the setting of Theorem \ref{sol.acc.decay}.
\begin{enumerate}
  \item
$\forall \varepsilon > 0 \ \exists t_1 > 0 \ \forall t > t_1 :$
 \[
\bigl| u_+(t,x) \bigr| \le \bigl(g(a_1,a_2,\beta)  + \varepsilon
\bigr) \ t^{-1/2}
 \]
for all $(t,x)$ satisfying
 \[ \sqrt{\frac{a_2 + \beta}{\beta}} \le \frac tx \le
    \sqrt{\frac{a_2 + \alpha}{\alpha}} \ .
 \]
  \item
 $ \forall \varepsilon > 0 \ \exists t_1 > 0 \ \forall t > t_1 :$
\begin{align*}
\bigl|\bigl|u_+(t,\cdot)\bigl|\bigl|_{L^2(I'_t)}^2&\le
\bigl(g(a_1,a_2,\beta)  + \varepsilon \bigr)^2 \ \Bigl(
\sqrt{\frac{\beta }{a_2 + \beta}}
-
\sqrt{\frac{\alpha }{a_2 + \alpha}}
\Bigr) \\
\end{align*}
  \item
\begin{align*}
\liminf_{t \rightarrow \infty}
\bigl|\bigl|u_+(t,\cdot)\bigl|\bigl|_{L^2(I'_t)}^2
&\le
 g(a_1,a_2,\beta)^2 \ \Bigl(
\sqrt{\frac{\beta }{a_2 + \beta}}
-
\sqrt{\frac{\alpha }{a_2 + \alpha}}
\Bigr) \\
&\sim
2\pi \beta(\sqrt{\beta}-\sqrt{\alpha}) \ a_2^{-1}
 \hbox{ as } a_2 \rightarrow \infty
\end{align*}
\end{enumerate}
\end{theo}
\begin{proof}

\noindent \underline{(i)}:

\noindent Theorem \ref{sol.acc.decay} implies that
$\forall \varepsilon > 0 \ \exists t_1 > 0 \ \forall t > t_1 :$
\begin{align*}
\bigl|u_+(t,x)\bigr|
&\le \bigl|u_+(t,x)- H(t,x,u_0)t^{-1/2}+H(t,x,u_0)t^{-1/2} \bigr| \\
&\le C(\psi,\alpha,\beta) t^{-1} + \bigl|H(t,x,u_0) \bigr|t^{-1/2} \\
& \le \bigl(g(a_1,a_2,\beta\bigr)  + \varepsilon) \ t^{-1/2}
\end{align*}
for $(t,x)$ in the cone indicated there.

\noindent \underline{(ii)}:

\noindent
 Follows from estimating the square of the $L^2$-norm against the square of the maximum of the
function (using (i)) times the length of the integration interval.

\noindent \underline{(iii)}:

\noindent Direct consequence of (ii).

\end{proof}

%
%


\begin{thebibliography}{99}
%
\bibitem{fam2} F. Ali Mehmeti, \it Spectral Theory and $L^{\infty}$-time
Decay Estimates for Klein-Gordon Equations on Two Half Axes with
Transmission: the Tunnel Effect. \rm Math. Methods Appl. Sci.
\textbf{17} (1994), 697--752.
\bibitem{ahr.arxiv2}  F. Ali Mehmeti, R. Haller-Dintelmann, V. R{\'e}gnier,
\it  Multiple tunnel effect for dispersive waves on a star-shaped
network: an explicit formula for the spectral representation. \rm
    arXiv:1012.3068v1 [math.AP], Preprint 2010.
%
\bibitem{L-infinity-proc-IWOTA}  F. Ali Mehmeti, R. Haller-Dintelmann, V. R{\'e}gnier,
\it The influence of the tunnel effect on $L^{\infty}-time$ decay.
\rm arXiv:1104.2993v1 [math.AP], to appear in: proceedings of IWOTA
2010, Preprint Valenciennes, 2011.
%
\bibitem{Alreg3} F. Ali Mehmeti, V. R\'egnier, \it Delayed reflection of the
energy flow at a potential step for dispersive wave packets. \rm Math. Methods
Appl. Sci. \textbf{27} (2004), 1145--1195.
\bibitem{vbl.evl}J.~von Below, J.A. Lubary, \it The eigenvalues of the Laplacian on
locally finite networks. \rm  Results Math. \textbf{47} (2005), no.
3-4, 199--225.
%
\bibitem{yas} Y. Daikh, \it Temps de passage de paquets d'ondes de basses
fr{\'e}quences ou limit{\'e}s en bandes de fr{\'e}quences par une barri{\`e}re
de potentiel. \rm Th{\`e}se de doctorat, Valenciennes, France, 2004.
%
\bibitem{D-L} J.M. Deutch, F.E. Low, \it Barrier Penetration and Superluminal
Velocity. \rm Annals of Physics \textbf{228} (1993), 184--202.
\bibitem{Nim} A. Enders, G. Nimtz, \it On superluminal barrier traversal. \rm
J. Phys. I France \textbf{2} (1992), 1693--1698.
\bibitem{hai.nim} A. Haibel, G. Nimtz, \it Universal relationship of time and
frequency in photonic tunnelling. \rm Ann. Physik (Leipzig) \textbf{10} (2001),
707--712.
\bibitem{kostr} V. Kostrykin, R. Schrader, \it The inverse scattering problem for metric graphs
and the travelling salesman problem. \rm Preprint, 2006
(www.arXiv.org:math.AP/0603010).
\end{thebibliography}
\end{document}